\documentclass[10pt]{amsart}
\usepackage{amssymb}

\usepackage{lmodern}

\usepackage{amsmath,amssymb,amscd,graphicx,epsfig}
\usepackage{lineno}
\usepackage{hyperref}
\usepackage{graphicx}
\usepackage{caption}
\everymath{\displaystyle}
\usepackage{enumerate}

\newtheorem{thm}{Theorem}[section]
\newtheorem{lemma}[thm]{Lemma}
\newtheorem{cor}[thm]{Corollary}
\newtheorem{prop}[thm]{Proposition}

\theoremstyle{remark}

\newtheorem{rem}{Remark}

\numberwithin{equation}{section}

\def\NN{\mathbb N}
\def\ZZ{\mathbb Z}

\def\RR{\mathbb{R}}

\def\B{\mathcal{B}}
\def\C{\mathcal{C}}

\def\E{\mathcal{E}}

\def\I{\mathcal{I}}
\def\J{\mathcal{J}}
\def\L{\mathcal{L}}

\def\U{\mathcal{U}}

\def\M{\mathcal{M}} 
\def\N{\mathcal{N}} 
\def\P{\mathcal{P}}

\def\S{\mathcal{S}}
\def\T{\mathcal{T}}

\usepackage{graphicx}

\newcommand{\norm}[1]{\Vert#1\Vert}

\newcommand{\abs}[1]{\lvert#1\rvert}

\newcommand{\diam}[1]{{\rm diam}({#1})}

\usepackage{multicol}

\begin{document}

\title{Linearly repetitive Delone sets are rectifiable}

\author{Jos\'e Aliste-Prieto}
 \address{ {\it J. Aliste-Prieto: }
Centro de Modelamiento Matem\'atico, Universidad de Chile, Blanco Encalada
2120 7to. piso,
Santiago, Chile.}
\email{jaliste@dim.uchile.cl}

\author{Daniel Coronel}
\address{ {\it D. Coronel: }Facultad de Matem\'aticas, Pontificia Universidad Cat\'olica de Chile, Campus San Joaqu\'in, Avenida Vicu\~{n}a Mackenna 4860, 
Santiago, Chile.} \email{acoronel@mat.puc.cl}

\author{Jean-Marc Gambaudo}
\address{ {\it J.-M. Gambaudo: }Laboratoire J.-A. Dieudonn\'e U.M.R. no 6621 du C.N.R.S.
Universit\'e de Nice - Sophia Antipolis
Parc Valrose
06108 Nice Cedex 02
France.} \email{gambaudo@unice.fr}




\date{\today}

\begin{abstract}
We show that  every linearly repetitive 
Delone set in the Euclidean $d$-space $\RR^d$, with $d\geq2$, 
is equivalent, up to a bi-Lipschitz homeomorphism, to the integer lattice
$\ZZ^d$. In the particular case when the Delone set $X$ in $\RR^d$  comes from
a primitive substitution tiling of $\RR^d$, we give a condition on the
eigenvalues of the  substitution matrix which ensures 
the existence of a homeomorphism with bounded displacement from $X$ to 
the lattice lattice $\beta\ZZ^d$  for some positive $\beta$. This condition
includes primitive Pisot substitution tilings but also concerns a much broader
set of substitution tilings. 
\end{abstract}

\maketitle

\section{Introduction}
Let $(Z,\delta)$ be a metric space. A subset $X$ of $Z$ is called a  \emph{Delone set} or \emph{separated net} if  there exist $r,R>0$ such that each ball
of radius $R$ in $Z$ contains at least one point of $X$ and each ball of radius
$r$ in $Z$ contains at most one point of $X$. 
Let $X_1$ and $X_2$ be two Delone sets in $Z$. We say that they are
\textit{bi-Lipschitz equivalent}
if there exists a homeomorphism $\phi:X_1\rightarrow X_2$ and a constant $K>0$
such that
\[\frac{1}{K}\delta(x,x')\le \delta( \phi(x),\phi(x')) \le K\delta(x,x')\]holds
for all $x$ and $x'$ in $X_1$. The map $\phi$ is then called a
\textit{bi-Lipschitz homeomorphism} between $X_1$ and $X_2$.
We say that a homeomorphism $\phi:X_1\rightarrow X_2$ is a \textit{bounded
displacement} if 
\[
 \sup_{x\in X_1} \delta(\phi(x),x) < \infty.
\]
Clearly a bounded displacement between two Delone sets 
is a bi-Lipschitz homeomorphism.

In the case when the ambient metric space $(Z,\delta)$ is  the $d$-dimensional
Euclidean space $\RR^d$, $d\ge 2$, endowed with the Euclidean distance, the
problem to know whether 
 two Delone sets are bi-Lipschitz equivalent was raised by Gromov in
\cite{Gromov}, 
and  boiled down
in Toledo's review \cite{Toledo} to the following question for the 2-dimensional Euclidean space: {\it Is every separated net in $\RR^2$ bi-Lipschitz equivalent to $\ZZ^2$?}
Counterexamples to this question were given  
independently  by Burago and Kleiner \cite{BK} and McMullen \cite{M}. McMullen also showed that when relaxing the bi-Lipschitz
condition to a H\"{o}lder one, all separated nets in $\RR^d$ are
equivalent. Later, Burago and Kleiner \cite{BK1}
gave a sufficient condition for a separated net to be bi-Lipschitz equivalent to $\ZZ^2$ 
and asked the following question:
{\it If one forms a separated net in the plane by placing a point in the
center of each tile of a Penrose tiling, is the resulting net bi-Lipschitz
equivalent to $\ZZ^2$?}
They studied the more general question of knowing whether a separated net arising 
from a cut-and-project 
tiling is bi-Lipschitz equivalent to $\ZZ^2$ (recall that the Penrose tiling is also a
cut-and project tiling \cite{dB}) and solved it in some cases that do not include the case of Penrose tilings, thus leaving the former question open.

More recently, Solomon \cite{Solomon} gave a positive answer in the case of Penrose tilings by using  the fact that they can be constructed using substitutions (see for instance \cite{GS}). 
In fact, Solomon proved that each separated net arising from 
a primitive substitution tiling in $\RR^2$ is bi-Lipschitz to $\ZZ^2$. 
Moreover, as an application of the work  of Laczkovich \cite{Lacz}, he showed that for every  substitution tiling of $\RR^d$  of  
Pisot type there is a bounded displacement between 
its associated separated net and $\beta\ZZ^d$ for some $\beta>0$ (see Section \ref{sec:defs} for more details).

During the same period and surprisingly rather independently, Delone sets in $\RR^d$, have been used in mathematical physics as models of solid materials.  In particular, after the discovery of quasicrystals at the beginning of the 
80's  \cite{shechtman:1984}, a strong impulse has been devoted to model these quasi-periodic materials by appropriate Delone sets, introducing in this way  the notion of ``repetitive'' Delone sets.

Later, Lagarias and Pleasants focused on  ``linearly repetitive'' Delone sets  \cite{LP} as a subclass of repetitive Delone sets that models all known examples of  quasicrystals. This class  includes all  Delone sets  arising from  self-similar  tilings (it contains in particular  the Penrose tiling drawn with triangles and the Penrose tiling drawn with ``thick'' and
``thin'' rhombi \cite{GS})  but is actually much broader (see \cite{S}, \cite{PS} and
\cite{S1}).

In this paper,  we make a connection between these two fields of research by using the second point of view (Delone sets and quasicrystals) to improve some known results concerning separated nets which are  bi-Lipschitz equivalent to  $\ZZ^d$ or obtained from  $\ZZ^d$  by a bounded displacement.

 On one hand, we prove that for any $d\geq 2$, every linearly repetitive Delone set in $\RR^d$ is bi-Lipschitz equivalent to $\ZZ^d$. On the other hand,  we show that  Delone sets  
 arising from a  class of substitution tilings of $\RR^d$, which is larger than the class of Pisot
 type  tilings,  are obtained from $\ZZ^d$ by a bounded displacement.

From now on, we will prefer the denomination ``Delone sets'', more widely used in the literature 
when the ambient metric space is the $d$-dimensional Euclidean space, than ``separated nets''. 
 
\section{Definitions and results}
\label{sec:defs}
\subsection{Repetitive Delone sets}

Let $d\geq 2$ and $X$ be a Delone set in $\RR^d$. We denote by $B(x,r)$ the closed ball around $x$ of radius $r$ in $\RR^d$. A set of the form $X \cap B(x,r)$ with $x\in X$ is called a \textit{patch (with size $r$)} 
of $X$ centered at $x$.
A Delone set $X$ is called \textit{repetitive} if for each $r>0$, there exists $M>0$
such that for each  point $z$ in $\RR^d$,  and for every patch with size $r$, $X \cap B(x,r)$, 
there exists $y$ in $X \cap B(z, M)$, such that:
\[X\cap B(y, r) = (X\cap B(x,r)) + y-x.\]
The smallest such $M$ is denoted by
$M_X(r)$ and it is called the \textit{repetitivity function} of $X$ (see \cite{LP}). 
If there exists $L>0$ such that $M_X(r)\le Lr$, then  $X$ is called \textit{linearly
repetitive}. 

Our first result is the following: 
\begin{thm}\label{thm:main}
 Every linearly repetitive Delone set in $\RR^d$ is bi-Lipschitz 
equivalent to
$\ZZ^d.$
\end{thm}
\begin{rem}Of course Theorem \ref{thm:main} is trivial when the dimension $d=1$
since, in this case, every Delone set (with no extra assumptions) is
bi-Lipschitz equivalent to $\ZZ$.
\end{rem}

\subsection{Substitution tilings}
Our second result concerns tilings arising from primitive substitutions. For more details about substitutions and tilings, 
see for instance \cite{S}. Let $d\geq2$ and $\Lambda$ be a closed subset of the Euclidean space $\RR^d$.  
A \emph{tiling} of  $\Lambda$ is a at most countable collection $\T =(t_j)_{j\in J}$ 
of closed subsets of $\Lambda$ that cover $\Lambda$ and have pairwise
disjoint interiors. 
The sets $t_j$ are called \emph{tiles} and, in this paper, all tiles 
are supposed to be homeomorphic to the unit closed ball in $\RR^d$. 
Tiles may also be \emph{colored}, which means that formally
they carry a label or color with them. There are several  notions of equivalence between tiles which  depend on the tilings under consideration. Let $\E$ be a group of isometries of $\RR^d$ containing all translations: two tiles $p$ and $q$ are \emph{$\E$-equivalent} (or in short $q$ is a $\E$-copy of $p$) if $q$ is the image of $p$ by an isometry in $\E$. If furthermore $p$ and $q$ are  colored, 
then they must have the same color.

Let $\P =\{p_1, \dots, p_k\}$ be a finite collection of tiles. 
A tiling $\T$ of  $\Lambda \subset \RR^d$ is \emph{$\E$-generated} by $\P$ 
if every tile in $\T$ is a $\E$-copy  of  some tile in $\P$. The tiles in $\P$ are called \emph{prototiles}. 
The set of all tilings of $\Lambda$ that are $\E$-generated by $\P$ is denoted 
by $\Omega_{\E,\P}(\Lambda)$. When $\Lambda = \RR^d$, we write $\Omega_{\E,\P}$ 
instead of $\Omega_{\E,\P}(\RR^d)$.

Given $\lambda > 1$ and a subset $U\subset \RR^d$, let $\lambda U:=\{\lambda x
\mid x\in U\}$ be the dilation of $U$ by $\lambda$. Similary, if $\M$ is a collection of tiles, the dilation of $\M$ by $\lambda$ is the 
collection  $\lambda \M:=\{\lambda t\mid t\in T\}$. Clearly, if $\M$ is a tiling, then $\lambda \M$ is also a tiling. A \emph{substitution rule} (with dilation factor $\lambda>1$) is a collection $\S
= (\S_{p_i})_{i=1}^k$, 
where for each $i\in\{1,\ldots,k\}$, $\S_{p_i}$ is a tiling of $\lambda p_i$ which is $\E$-generated by 
$\P$. Thus, each $\S_{p_i}$ gives the rule of  how to decompose $\lambda p_i$ into prototiles, for each $i$. 
Every substitution induces a natural map $\I_\S$ on $\Omega_{\E,\P}$, which first dilates tiles by $\lambda$ and then 
replaces dilated tiles by a patch of prototiles according to the substitution rule, for more details see Section \ref{Nous}.

A tiling $\T$ in $\Omega_{\E,\P}$ is said to be \emph{admissible} for $\S$ if it belongs to
\[\Omega_{\S} := \bigcap_{k\geq 0}\I_{\S}^k(\Omega_{\E, \P}).\]

For any given substitution, we can associate an integer matrix, in which each element counts how many tiles of a given type appear in a tile of another type after dilation and substitution. Depending on the definition of tile-type, we may obtain different matrices. For our purposes here, we consider the following definition (see also \cite{Solomon}): Given a substitution rule $\S$, we say that two tiles $p_i$ and $p_j$ have the same \emph{type} (or are \emph{$\S$-equivalent}) if there exists an isometry $O$ of $\mathbb{R}^d$ such that $O(p_i) = p_j$ and $O(\S_{p_i}) = \S_{p_j}$. Let $\mathcal{Q}=\{q_1,\ldots,q_n\}$ be the set of tile-types of all prototiles. The \emph{substitution matrix}
is then defined as the $n\times n$ matrix $M_{\S}=(m_{i,j})_{i,j}$, where each $m_{i,j}$ is the number of tiles of type $q_j$ that belong to $\S_{p}$ where $p$ is any prototile of type $q_i$. The definition of tile-type implies that $m_{i,j}$ does not depend on $p$ and thus $M_\S$ is well-defined. 

Finally, we recall some basic definitions of Perron-Frobenius theory needed to state our result. A matrix $M$ is \emph{primitive} if there exists $n>0$ such that all the  elements of $M^n$ are positive. By  Perron-Frobenius theorem, every primitive matrix $M$ 
has a largest positive real eigenvalue $\mu$, the \emph{Perron eigenvalue}, and moreover it has no other eigenvalue with the same modulus as $\mu$. It is easy to check that the Perron eigenvalue of $M_\S$ for a given substitution rule $\S$ is $\mu = \lambda^d$ when $M_\S$ is primitive.

Given a tiling tiling $\T$ in  $\Omega_{\E,\P}$, define $X_\T$ to be the set of barycenters of all the tiles in $\T$. It is clear
that $X_\T$ is a Delone set in $\RR^d$, and we call it  the  \textit{Delone set induced by 
$\T$}.
\begin{thm}\label{thm:2}
Suppose that  $\S$ is a substitution rule with dilation factor $\lambda$,
 that the substitution matrix $M_\S$ is primitive and that 
\[r(M) := \max \{|\eta|\mid \eta \neq \mu\text{ is an eigenvalue of } M \} < \lambda,\]
where $\mu$ is the Perron eigenvalue of $M_\S$. Then there exists $\beta>0$ such that  for every tiling
$\T$ in $\Omega_\S$, the Delone set $X_\T$
induced by $\T$ is obtained from $\beta\ZZ^d$ by a bounded displacement.
\end{thm}
\begin{rem}
A substitution rule $\S$ is of \emph{Pisot
type} if $r(M_\S) < 1$ (compare with the definition in \cite{Solomon}).  
\end{rem}
We finish this section by introducing some useful notations. For any $d\geq 1$, we denote by $\mu_d$ the Lebesgue measure in $\RR^d$. For every subset $U$ of $\RR^d$, we denote its boundary  by $\partial U$.  Furthermore, for every  subset $U$ of $\RR^d$  and every  Delone set  $X$ in $\RR^d$, we denote by  $\N(X,U)$ the number of points of $X$ inside $U$.  
\section{Proof of Theorem \ref{thm:main}}

In \cite{BK1}, Burago and Kleiner gave a sufficient condition for a Delone set in $\RR^2$ to be bi-Lipschitz equivalent to $\ZZ^2$. This condition concerns  the speed of  convergence to an asymptotic density of the number of points of $X$ inside larger and larger balls. As we will see, it turns out that an analog condition works in every dimension $d\geq 2$. First, we need some definitions. A  \textit{cube C with size $l(C) >0$} in $\RR^d$ is a subset of the form  
$C=x+[0,l(C)]^d$, where $x \in \RR^d$. Let  $X\subset \RR^d$ be a Delone set. Given $\rho>0$ 
and a cube $C\subset \RR^d$  with size $l(C)$, define $e_\rho(C)$ to be the
\emph{density deviation} 
\begin{equation}
e_\rho(C) =  \max \left(\frac{\rho \mu_d(C)}{\N(X,C)} , \frac{\N(X,C)}{\rho  \mu_d(C)}
\right).
\end{equation}
Next, for $k\in\NN$, define  $E_\rho(k)$ as the supremum
of the quantities $e_\rho(C)$, where $C$ ranges over all cubes with size $l(C)
=k$ and vertices at $\ZZ^d$. 
 The condition reads as follows. 
\begin{thm}\label{d>1} 
Let  $d \geq 2$ and   $X\subset \RR^d$ be a Delone set. Suppose that there is  
$\rho>0$ such that the product
\[\prod_{m=1}^{+\infty} E_\rho(2^m)\] converges.  Then $X$ is bi-Lipschitz equivalent to $\ZZ^d$. 
\end {thm}
\begin{rem} When $d=2$, Theorem \ref{d>1}  corresponds to the main theorem of \cite{BK1}. 
In the proof of \cite{BK1}, the authors solve  a  prescribed volume form equation. 
To prove this result for $d>2$  we use a very useful construction by Rivi\`ere and Ye \cite{RY}, 
which  actually simplifies  the original proof of \cite{BK1}. 
The proof  of Theorem \ref{d>1} is given in the Appendix A (Section \ref{Dong}).
\end{rem}

Linearly repetitive Delone sets are good candidates to satisfy the condition of Theorem \ref{d>1} as suggested by the following result of Lagarias and Pleasants: 
\begin{thm}\label{thm:LP}\cite{LP} 
Let  $d\geq 2$ and $X$ be a linearly repetitive Delone set  in $\RR^d$  Then there exist 
positive constants $ \rho(X)$ and $\delta(X)<1$,  such that   for any cube $C$  with size $l(C)$, we have:
\begin{equation}\label{eq}|\N(X,C) - \rho(X)\mu_d(C)| = O(l(C)^{d-\delta(X)}). 
\end{equation}
\end{thm}
\begin{rem} Lagarias and Pleasants proved a stronger version of the above theorem by giving similar estimates for the occurrences of every patch in $X$. We will not use this stronger version here. 
\end{rem}

\proof[Proof of Theorem \ref{thm:main}]
Let $X$ be a linearly repetitive Delone set in $\RR^d$.
By Theorem \ref{d>1}, it is enough to show that there
exists  $\rho>0$ such that the product  $\prod_{m=1}^\infty E_\rho(2^m)$  converges. 
Indeed, Theorem \ref{thm:LP}, tells us that there exists $\rho(X), M, \delta$ and $l_0$, all positive, such that 
\[
\left|\frac{\N(X, C)}{\mu_d(C)}-\rho(X)\right|\le Ml(C)^{-\delta},
\]
for  every cube  $C$ in $\RR^d$ with side  $l(C)\ge l_0$. 
Since $\rho(X)>0$, a simple computation shows that there exist constants $M', l_1 > 0$ such that 
\[
\max \left(\left|\frac{\rho(X)\mu_d(C)}{\N(X,C)}-1\right|,\left|\frac{\N(X, C)}{\rho(X)\mu_d(C)}-1\right|\right)\le M'
l(C)^{-\delta} 
\]
for every  cube  $C$ with side $l(C)\ge l_1$ and thus
\[1\le e_{\rho(X)}(C)\le 1+ M' l(C)^{-\delta}.\]
Taking the supremum we get
\begin{equation}\label{eq:1} 
1\le E_{\rho(X)}(l(C)) \le 1+ M' l(C) ^{-\delta},\, \textrm{for all }  l(C)\ge l_1. 
\end{equation}
It follows that $\sum_{m=1}^{\infty}\log E(2^m)$ converges which  
implies that $\prod_{m=0}^{\infty}E(2^m)$ also converges and the conclusion now follows.\endproof

\section{Proof of Theorem \ref{thm:2}}

First let us introduce some notations. A {\it unit cube} in $\RR^d$ is a cube with size 1 whose vertices have integer coordinates. We denote by $\U$ the set of all subsets of $\RR^d$ which coincide with a finite union of  unit cubes.  Similarly, for all $\delta>0$, we denote by $\U_\delta$  the set of all subsets of $\RR^d$ which coincide with a finite union of cubes with size $\delta$ whose vertices have coordinates in $\delta\ZZ^d$  . In \cite{Lacz}, Laczkovich obtained the following very elegant characterization
of Delone sets that can be obtained from a lattice $\beta\ZZ^d$ by a bounded
displacement.

\begin{thm}\label{thm:Laczkovich} \cite{Lacz}
For any $d\geq2$ and for  every Delone set $X$ in $\RR^d$ and  every $\alpha > 0,$ 
the following statements are equivalent:
\begin{enumerate}
 \item There exists $K>0$ such that for  every subset $U \in \U$, 
\begin{equation}\label{eq:lac}
\abs{\N (X,  U) - \alpha \mu_d(U)} \le K \mu_{d-1}(\partial U); 
\end{equation}

\item There is a bounded displacement from $X $ onto $ \alpha^{-1/d} \ZZ^d$.
\end{enumerate}
\end{thm}

\begin{rem}\label{rescaling}Notice that by rescaling Theorem
\ref{thm:Laczkovich} works as well if we prove Equation \eqref{eq:lac} for all
subsets $U$ in some $\U_\delta$.
\end{rem}
Using Laczkovich's characterization and the above remark, Theorem \ref{thm:2} turns to be a straightforward
corollary of the following result, whose  proof follows arguments introduced in \cite{ACG} and is given in Appendix B, Section \ref{Nous}.


\begin{thm}\label{lem:key}
Let $ d>0$ and let  $\S$ be a  substitution  rule with dilation factor $\lambda$ and 
primitive substitution matrix $M_\S$. If $r(M_\S)< \lambda$, then  there exist $\delta>0$, 
$K>0$ and $\alpha > 0$ such that for every tiling $\T$ in $\Omega_\S$,
\begin{equation}\label{eq.cubosconexos}
\abs{\N (X_\T,  U) - \alpha \mu_d(U)} \le K \mu_{d-1}(\partial U)
\end{equation} 
for every  subset $U\in \U_\delta$.
\end{thm}

\section{Appendix A: proof of Theorem \ref{d>1} \label{Dong} }
Burago and Kleiner proved that Theorem  \ref{d>1} is a consequence of the following proposition that they prove in dimension 2.
\begin{prop}\label{BuKl}
Let $ u:\RR^d\to \RR$ be a positive function which is constant on each open unit
cube with vertices in $\ZZ^d$, and  let $\rho >0$ be given. 
Assume that $u$ and $1/u$ are bounded.
For any cube $C$  in
$\RR^d$, let $e(C)$ be the quantity
\[\max \left\{\frac{\rho}{\frac{1}{\vert C\vert}\int_Cu}, \frac{\frac{1}{\vert C\vert}\int_Cu}{\rho}\right\},\]
where $\vert C\vert =\mu_d(C) $ is the Lebesgue measure of the cube $C$. Define
an ``error'' function 
$E: \NN \to \RR$ by letting $E(k)$ be the supremum of $e(\cdot)$ taken over the
collection of cubes of the form $[i_1, i_1+k] \times \dots[i_d, i_d+k]$, where
$(i_1, \dots, i_d) \in \ZZ^d$.
If the product 
\[\prod_{i=1}^{+\infty} E(2^i)\]
converges, then there exists a bi-Lipschitz homeomorphism $\Psi:\RR^d\to \RR^d$
with $\mathrm{det} (\nabla \Psi) =  u \,\,\, a.e.$.  
\end{prop}
We prove this proposition in any dimension. Our proof  is shorter than the specific one done in the two dimensional case in   \cite{BK1} and relies on the following lemma\footnote{Actually $(iv)$ is not explicitly written in \cite{RY} but turns to be  a straightforward consequence of the construction.} proved by D. Ye and T. Rivi\`ere:
\begin{lemma}\cite{RY}\label{RiYe}
Let $D= [0, 1]^d$, $A = [0, 1]^{d-1}\times [0, 1/2] $ and $B = [0, 1]^{d-1}\times [1/2, 1] $ . Let $\alpha>0$ and $\beta >0$ such that $\alpha +\beta =1$. There exists a bi-Lipschitz homeomorphism $\Phi$  from $D$ into itself such that:
\begin{itemize}
\item [(i)] $\Phi(x) =x, \quad x\in \partial D$;
\item [(ii)] $\det (\nabla \Phi) \equiv 2 \alpha$ in $A$ and $\det (\nabla \Phi) \equiv 2 \beta$ in $B$;
\item [(iii)] $\Vert \nabla (\Phi -Id)\Vert_{L^\infty(D)}\leq   C_\eta\vert \beta- \alpha\vert $;
\item [(iv)] $\Vert \nabla (\Phi^{-1} -Id)\Vert_{L^\infty(D)}\leq  C_\eta\vert \beta- \alpha\vert $,
\end{itemize}
where $0<\eta\leq\alpha$ , $0<\eta\leq\beta$ and $C_\eta$ only depends on $\eta$.

\end{lemma}

In order to prove Proposition \ref{BuKl}, we adapt to our context a constructive method  developed in \cite{RY}.  Let $\bar n = (n_1, \dots, n_d)$ be a point in $\ZZ^d$ and $m>0$ be a positive integer. Consider the cube:
\[C_{\bar n , m}\, =\, \prod_{l= 1}^{d} [n_l, n_l +2^{m}].\]
In each of these cubes, for each integer $i$, $0\leq i\leq m,$ and each integer vector 
$\bar k = (k_1, \dots, k_d)$ in $\Lambda_{m,i}:=\ZZ^d\cap\prod_{l=1}^{d}[0, 2^{m-i})$, 
we consider the smaller cube:
\[C_{\bar n , m, i, \bar k}\, =\, \prod_{l= 1}^{d} [n_l +k_l2^{i} , n_l +(k_l+1)2^{i}].\]
Notice that for each $0\leq i\leq m,$ and every pair  $\bar k \neq \bar k'$ ,the cubes $C_{\bar n , m, i, \bar k}$  and $C_{\bar n , m, i, \bar k'}$ have disjoint interiors. Moreover, 
\[\bigcup_{\bar k\in  \Lambda_{m,i}} C_{\bar n , m, i, \bar k} \,=\, C_{\bar n , m}.\]

Take $\epsilon = (\epsilon_1, \dots, \epsilon_d) \in \{0, 1\}^d$, let $1\leq p\leq d$ and denote by $A^p_{\bar n , m, i, \bar k}(\epsilon)$, $B^p_{\bar n , m, i, \bar k}(\epsilon)$ and $D^p_{\bar n , m, i, \bar k}(\epsilon)$ the following subsets of 
$C_{\bar n , m, i, \bar k}$:

\begin{multline*}
A^p_{\bar n , m, i, \bar k} (\epsilon) =
\prod_{l =1}^{p-1} [n_l +k_l2^{i} , n_l +(k_l+1)2^{i}]\\\times  [n_p+k_p2^{i} , n_p +(k_p+\frac{1}{2})2^{i}]\times 
\prod_{l =p+1}^{d} [n_l +(k_l +\frac{\epsilon_l}{2})2^{i} , n_l +(k_l+\frac{\epsilon_l+1}{2})2^{i}],
\end{multline*}
\begin{multline*}
B^p_{\bar n , m, i, \bar k} (\epsilon) =
\prod_{l =1}^{p-1} [n_l +k_l2^{i} , n_l +(k_l+1)2^{i}]\\\times  [n_p+(k_p+\frac{1}{2})2^{i} , n_p +(k_p+1)2^{i}]\times 
\prod_{l =p+1}^{d} [n_l +(k_l +\frac{\epsilon_l}{2})2^{i} , n_l +(k_l+\frac{\epsilon_l+1}{2})2^{i}],
\end{multline*}
and 
\[ D^p_{\bar n , m, i, \bar k}(\epsilon) = 
A^p_{\bar n , m, i, \bar k}(\epsilon) \cup B^p_{\bar n , m, i, \bar k}(\epsilon).\]
Let  $ u:\RR^d\to \RR$ be a positive function which  is constant on each open unit cube with vertices in $\ZZ^d$ and consider  the two positive numbers
$$\alpha^p_{\bar n , m, i, \bar k} (\epsilon) = 
\frac{ \int_{A^p_{\bar n , m, i, \bar k} (\epsilon) }u(x) dx}{ \int_{D^p_{\bar n , m, i, \bar k} (\epsilon) }u(x) dx},$$
and 
$$\beta^p_{\bar n , m, i, \bar k} (\epsilon) =\frac{ \int_{B^p_{\bar n , m, i, \bar k} (\epsilon) }u(x) dx}{ \int_{D^p_{\bar n , m, i, \bar k} (\epsilon)}u(x) dx}.$$
We clearly have 
$$\alpha^p_{\bar n , m, i, \bar k} (\epsilon) + \beta^p_{\bar n , m, i, \bar k} (\epsilon)  =1.$$

\noindent 
\begin{lemma} 
$\forall\epsilon\in \{0, 1 \}^d,\,   \forall p \in \{1, \dots, d\},\,   \forall  \bar k \in \Lambda_{m,i},   \forall  i \in \{1, \dots , m\},\,  \forall m>0, $\,  and \, $ \forall \bar n \in \ZZ^d,$ we have:

\[\frac{1}{2}\frac{1}{(E(2^{i-1}))^2} \leq \alpha^p_{\bar n , m, i, \bar k} (\epsilon) \leq \frac{1}{2}{(E(2^{i-1}))^2}\]
and 
\[\frac{1}{2}\frac{1}{(E(2^{i-1}))^2} \leq \beta^p_{\bar n , m, i, \bar k} (\epsilon) \leq \frac{1}{2}{(E(2^{i-1}))^2}.\]

Consequently, if $u$ satisfies the hypothesis of Proposition \ref{BuKl}, 
then there exists $0<\eta^\star< 1$  such that 
\[\eta^\star \leq \alpha^p_{\bar n , m, i, \bar k} (\epsilon)\]
and 
\[\eta^\star \leq \beta^p_{\bar n , m, i, \bar k} (\epsilon).\]
\end{lemma}
\begin{proof}
The very definition of $E(2^i)$ shows that for all  $C_{\bar n , m, i, \bar k}$ we have:
\[
 \frac{1}{E(2^{i})}\rho \vert C_{\bar n , m, i, \bar k}\vert    
\leq \int_{C_{\bar n , m, i, \bar k}} u(x) dx\leq  E(2^{i})\rho \vert C_{\bar n , m, i, \bar k}\vert.
\]
These estimates remain true for a given $i$,  for every finite collection $R_i$ of cubes $C_{\bar n , m, i, \bar k}$   with disjoint interiors:
\[
 \frac{1}{E(2^{i})}\rho \vert R_i\vert    \leq \int_{R_i} u(x) dx\leq  E(2^{i})\rho \vert R_i\vert.
\]
We remark that  when $1\leq i\leq m$, $A^p_{\bar n , m, i, \bar k} (\epsilon) $ and $B^p_{\bar n , m, i, \bar k} (\epsilon) $  are precisely finite union of cubes 
$C_{\bar n , m, i-1, \bar k}$ with disjoint interiors. The proof of the lemma follows easily. 
\end{proof}

Up to  homothety and rotation, we are in a situation to apply Lemma \ref{RiYe}  and get a bi-Lipschitz homeomorphism  
$\Phi^p_{\bar n , m, i, \bar k} (\epsilon) $ from $D^p_{\bar n , m, i, \bar k} (\epsilon)$ into itself such that: 
\begin{enumerate}
\item [(i)] $\Phi^p_{\bar n , m, i, \bar k} (\epsilon) (x) =x, \quad x\in \partial D^p_{\bar n , m, i, \bar k} (\epsilon) $;
\item [(ii)] $\textrm{det} (\nabla\Phi^p_{\bar n , m, i, \bar k} (\epsilon) ) = 2 \alpha^p_{\bar n , m, i, \bar k} (\epsilon)$ in $A^p_{\bar n , m, i, \bar k} (\epsilon)$;
\item [(iii)] $\textrm{det} (\nabla \Phi^p_{\bar n , m, i, \bar k} (\epsilon) ) = 2 \beta^p_{\bar n , m, i, \bar k} (\epsilon)$ in $B^p_{\bar n , m, i, \bar k} (\epsilon)$;
\item [(iv)] $\Vert \nabla (\Phi^p_{\bar n , m, i, \bar k} (\epsilon) -Id) \Vert_{L^\infty(D^p_{\bar n , m, i, \bar k} (\epsilon))}\leq    \frac{1}{2}C_{\eta^\star} ((E(2^{i-1}))^2-(E(2^{i-1}))^{-2})$;
\item [(v)] $\Vert \nabla (\Phi^p_{\bar n , m, i, \bar k} (\epsilon)^{-1} -Id) \Vert_{L^\infty(D^p_{\bar n , m, i, \bar k} (\epsilon))}\leq   \frac{1}{2}C_{\eta^\star}((E(2^{i-1}))^2-(E(2^{i-1}))^{-2}) .$
\end{enumerate}
We denote by $\Phi^p_{\bar n , m, i, \bar k} $ the bi-Lipschitz homeomorphism from 
$C_{\bar n , m, i, \bar k}$ into itself defined by 
$\Phi^p_{\bar n , m, i, \bar k} (x) = \Phi^p_{\bar n , m, i, \bar k} (\epsilon) (x) $  
if $x\in D^p_{\bar n , m, i, \bar k} (\epsilon)$, 
and consider the bi-Lipschitz homeomorphism obtained by composition:
\[
\Phi_{\bar n , m, i, \bar k}  \,=\, \Phi^d_{\bar n , m, i, \bar k}\circ\Phi^{d-1}_{\bar n , m, i, \bar k}\circ \dots \circ \Phi^1_{\bar n , m, i, \bar k}. 
\]
Since the Jacobians of the  $\Phi^p_{\bar n , m, i, \bar k} $ are constant by parts, a simple calculation shows that
\[
 \textrm{det} (\nabla \Phi_{\bar n , m, i, \bar k})  \,= \,2^d \frac{ \int_{C_{\bar n , m, i-1, \bar k'}}u(x) dx}{ \int_{C_{\bar n , m, i, \bar k}}u(x) dx},
\]
for any sub-cube $C_{\bar n , m, i-1, \bar k'}$ in $C_{\bar n , m, i, \bar k}$. It is also direct to check that: 
\[
\Vert \nabla  \Phi_{\bar n , m, i, \bar k} \Vert_{L^\infty( C_{\bar n , m, i, \bar k})}\,\leq K_{i, d} \quad{\rm{and}}\quad 
\Vert \nabla  \Phi^{-1}_{\bar n , m, i, \bar k} \Vert_{L^\infty( C_{\bar n , m, i, \bar k})}\,\leq K_{i, d},
\]
where \[K_{i, d} = (1+ \frac{1}{2}{C_{\eta^\star}} ((E(2^{i-1}))^2-(E(2^{i-1}))^{-2}))^d.\]

We denote by $\Phi_{\bar n , m, i} $ the bi-Lipschitz homeomorphism from $C_{\bar n , m }$ 
into itself defined by $\Phi_{\bar n , m, i} (x) = \Phi_{\bar n , m, i, \bar k}  (x) $  
if $x\in C_{\bar n , m, i, \bar k}$, and we consider 
 the bi-Lipschitz  homeomorphism obtained by composition:
 \[
  \Psi_{\bar n , m } \,=\,\Phi_{\bar n , m, m}\circ\Phi_{\bar n , m, m-1 } \circ \dots \circ \Phi_{\bar n , m, 1}.
 \]
 The map $\Psi_{\bar n , m }$ is a bi-Lipschitz homeomorphism on $ C_{\bar n , m}$ which satisfies:
 \begin{itemize}
 \item  $\Psi_{\bar n , m } = Id $  on $\partial C_{\bar n , m}$;
 \item $\textrm{det} (\nabla \Psi_{\bar n , m }) 
         =   \frac{u}{\dfrac{1}{|C_{\bar n , m}|}\int_{C_{\bar n , m}}u(x) dx}$ in each unit cube with integer vertices in $ 
 C_{\bar n , m}$;
 \item $\Vert \nabla  \Phi_{\bar n , m}\Vert_{L^\infty( C_{\bar n , m})}\,\leq  
\prod_{i=1}^{m}K_{i, d} $\, and \,  
$\Vert \nabla  \Phi_{\bar n , m}^{-1}\Vert_{L^\infty( C_{\bar n , m})}\,\leq  \prod_{i=1}^{m}K_{i, d}.$

 \end{itemize}
 
\begin{lemma}\label{bof}
If the product 
$\prod_{i = 1}^{+\infty}  E(2^i)$
converges, then the product  $\prod_{i = 1}^{+\infty} K_{i, d}$ converges too.
\end{lemma} 
\begin{proof} It is enough to show that the product 
$$\prod_{i = 1}^{+\infty}  (1+ \frac{1}{2}{C_{\eta^\star}} ((E(2^{i-1}))^2-(E(2^{i-1}))^{-2}))$$
converges and this last product converges if the series:
$$\sum_{i = 1}^{+\infty} \ln (1+ \frac{1}{2}{C_{\eta^\star}} ((E(2^{i-1}))^2-(E(2^{i-1}))^{-2}))$$ converges. This will be true if  
$$\sum_{i = 1}^{+\infty} (E(2^{i-1}))^2-(E(2^{i-1}))^{-2}) <+\infty,$$ and thus if 
$$(\star) \quad \sum_{i = 1}^{+\infty} (E(2^{i-1}))-(E(2^{i-1}))^{-1}) <+\infty.$$
On the one hand, the convergence of  the product $\prod_{i = 1}^{+\infty}  E(2^i)$
implies  that the series
$$\sum_{i = 1}^{+\infty} \ln (1+ (E(2^{i-1}) -1)) $$   converges which in turn implies that
$$\sum_{i = 1}^{+\infty}  (E(2^{i-1}) -1) <+\infty. $$
On the other hand, for $i_0$ large enough
$$\sum_{i = i_0}^{+\infty} (E(2^{i-1}))-(E(2^{i-1}))^{-1}) \leq 2\sum_{i = i_0}^{+\infty} (E(2^{i-1}))-1)$$
 \end{proof}
 
 We extend each  $\Psi_{\bar n , m } $ by the Identity out of $ C_{\bar n , m}$ to get a bi-Lipschitz homeomorphism of $\RR^d$.

 It follows from Lemma \ref{bof} that the sequence of bi-Lipschitz constants of the  bi-Lipschitz homeomorphisms  $\Psi_{\bar n , m } $ 
 is bounded and thus, using Azerl\`a-Ascoli theorem, there exists an accumulation point $\Phi$ which is a bi-Lipschitz homeomorphism on $\RR^d$ and which satisfies: 
$$\textrm{det} (\nabla \Phi) = \frac{u}{\rho}$$ in each unit cube with integer vertices in $\RR^d$.  By post-composing $\Phi$ with  a homothety with ratio 
$\rho^{\frac{1}{d}}$ we get an  answer to Proposition \ref{BuKl}.

\vspace{0.4cm}

The reminder of the proof of Theorem  \ref{d>1} , which uses the Hall marriage lemma, follows exactly  \cite{BK1} 
and the dimension is irrelevant, see Section 4 in  \cite{BK1}.  
\section{Appendix B: Proof of Theorem \ref{lem:key}\label{Nous}}
We first state Theorem \ref{lem:key1} below, which is a generalization of the main theorem in \cite{ACG}. 
\begin{thm}\label{lem:key1}
Let $ d\geq 2$ and let  $\S$ be a  substitution  rule with dilation factor $\lambda$
and primitive substitution matrix $M_\S$. If $r(M_\S)< \lambda$, then  there
exist $\delta>0$, $K>0$ and $\alpha > 0$ such that for every tiling $\T$ in $\Omega_\S$,
\[
\abs{\N (\T,  U) - \alpha \mu_d(U)} \le K \L(\T, \partial U)
\] 
for every  $U\in \U_\delta$, where $\N (\T,  U)$ stands for the the number of tiles of $\T$ that are contained in $U$ and  $\L(\T, \partial U)$  stands for the number of tiles of $\T$ that intersect $\partial U$.  
\end{thm}
Let  us show how to prove Theorem \ref{lem:key} assuming Theorem \ref{lem:key1}. We  need to compare the number of tiles that intersect
the boundary of a subset $U\in \U_\delta$  with the measure of its boundary. 

For  a subset $A$ of $\RR^d$, define
\[r_\text{min}(A) := \sup\{r>0 \mid \text{there exists } x\in A  \text{ s.t.
}B(x,r)\subset A\}.\]

Given a tiling $\T$ of $\RR^d$, define
\[r_\T=\inf\{r_\text{min}(t)\mid t\in\T\}\]  
and 
\[R_{\T}=\sup\{\diam{t}\mid{t\in\T}\}/2.\]
If  $0<r_\T< R_\T<+\infty$, then we say that $\T$ is \emph{locally finite} and 
define $K_\T:=\lfloor 4^dR_\T^d r_\T^{-d}\rfloor$.  Notice that the substitution tilings we consider in this paper are locally finite. 
A simple computation yields
(see \cite{ACG} for details):
\begin{lemma} \label{bound}Let $d\geq 2$ and $\T$ be a locally finite tiling  of
$\RR^d$, then each ball of radius smaller than  or equal to $2R_\T$ meets at
most  $K_\T$ tiles of $\T$.  
\end{lemma}

\begin{cor}\label{cor}
Let $d\geq2$, $\T$ be a locally finite  tiling of $\RR^d$ and $\delta>0$. Then, for every
subset  $U \in \U_\delta$ we have:
\[\L(\T, \partial U) \leq K_\T \left(\frac{\sqrt{d-1} }{2R_\T} +1\right)^{d-1}  
\delta^{-(d-1)}\mu_{d-1}(\partial U).\]
\end{cor}
\begin{proof} Each facet of $\partial U$ is covered by less than  
$(\lfloor\sqrt{d-1}/(2R_\T) \rfloor+1)^{d-1}$ $d-1$-cubes of size $
{2R_\T}/{\sqrt{d-1}}$ having pairwise disjoint $d-1$-interiors. Each of these
cubes is included in a $d$-ball of radius $2R_\T$ and thus intersect a most
$K_\T$ tiles of $\T$. The number of facets of $ \partial U$ is equal to
$\delta^{-(d-1)}\mu_{d-1}(\partial U).$\end{proof}

Since there is exactly one point of
$X_\T$ in the center of each tile of $\T$, it follows that  
\begin{equation}
 0 \le \N(X_\T,U) - \N(\T,U) \leq \L(\T,\partial U).
 \end{equation}

Combining this inequality with  Theorem \ref{lem:key1} and Corollary \ref{cor}
 we obtain that there exists a constant $K>0$ such that
\begin{equation}
|\N(X_\T,U) - \alpha \mu_d(U)| < K\mu_{d-1}(\partial U)
\end{equation}
for every  for every $U\in \U_\delta$, which proves Theorem \ref{lem:key}. 

We now turn to the proof of Theorem \ref{lem:key1}].  First let us choose the 
adequate scale $\delta$. Let $\T$ be a tiling in $\Omega_\S$. We say that
$\delta$ {\it fits with} $\T$ if it is chosen large enough such that for any
subset $U\in \U_\delta$:
\begin{itemize}
\item $U$ contains a tile of $\T$;

\item for any connected component  $\C$ of $\partial U$, $K_\T< \N(\T, \C)$;
\item two distinct connected components of $\partial U$ cannot intersect the same tile of $\T$. 
\end{itemize}
From now, $\delta$ will be chosen to fit with $\T$. The proof of Theorem \ref{lem:key1} is divided in four key arguments. 

\noindent $\bullet$  The first argument is a topological simplification. 
We show that it suffices to consider $U$ in $\U_\delta$ 
such that $U$ and $\partial U$ are connected. Indeed,  
for any $U\in \U_\delta$, there is a finite collection $V_1, \dots, V_n$ of  
pairwise disjoint connected elements in $\U_\delta$ such that:
$$ U = \cup_{i=1}^{n} V_i.$$
We denote by $\partial V_{i, 0}, \dots \partial V_{i, p(i)}$ the connected components of $\partial V_i$, and fix   $\partial V_{i, 0}$ to be the connected component that bounds the component of the complementary of $V_i$ with infinite diameter. For $j= 1, \dots, p(i)$,  $\partial V_{i, j}$ is the boundary of an element  $V_{i,  j} \in \U_\delta$ which is connected  and whose interior does not  intersect $V_i$. We denote by $\hat V_i$ the union:
$$\hat V_i = V_i \cup V_{i, 1}\cup \dots\cup V_{i, p(i)},$$
it belongs to $\U_\delta$ is connected and its boundary $\partial V_{i, 0}$ is connected. 
  If Theorem  \ref{lem:key1} works for a  $\delta$ that fits with $\T$ and all $U\in \U_\delta$ which are connected and have a connected boundary, we obtain that there exists a constant $K>0$ such that, on  one hand:
\begin{equation}
|\N(\T,\hat V_i) - \alpha \mu_d(\hat V_i)| < K\L(\T, \partial V_{i, 0}),\quad \forall i\in \{1, \dots, n\},
\end{equation}
and on the other hand:
\begin{equation}
|\N(\T,V_{i, j}) - \alpha \mu_d(V_{i, j})| < K\L(\T, \partial V_{i, j}),\quad \forall i\in \{1, \dots, n\}, \quad \forall j\in \{1, \dots,p(i)\}.
\end{equation}
Since $$\N(\T,\hat V_{i}) =  \N(\T, V_{i}) + \sum_{j=1}^{p(i)}(\N(\T, V_{i,j}) +
\L(\T, \partial V_{i,j}))$$
and 
$$\mu_d(\hat V_{i}) = \mu_d(V_{i})+ \sum_{j=1}^{p(i)}\mu_d(V_{i,j}),$$
we get 
$$
|\N(\T, V_i) - \alpha \mu_d( V_i)| < (K+1)\left(\sum_{j=0}^{p(i)} \L(\T,
\partial V_{i, j})\right),\quad \forall i\in \{1, \dots, n\}.
$$
Summing on all the connected components of $U$ we get 
$$|\N(\T, U) - \alpha \mu_d( U)| < (K+1)\L(\T, \partial U).$$
From now we will restrict ourselves to the case when $U$ and $\partial U$ are connected. The remainder of the proof follows the same lines
of the proof of the main result in \cite{ACG}. 

\noindent $\bullet$
The second argument is a geometrical one and is developed in the following two lemmas. These lemmas are proved in \cite{ACG} for the two-dimensional case. The generalization of the first one to higher dimension is straightforward and therefore we skip it, while the proof of the second one requires slight modifications so is given for the convenience of the reader.
\let\Tt=\T
\begin{lemma}
\label{lem:diam_jordan_arc}
Let $\Tt$ be a locally finite tiling  and $\gamma$ be a simple curve.  
Then, 
\[ \operatorname{diam}(\gamma)\leq 2R_\Tt \L(\T, \gamma),\]
where  $\L(\T, \gamma)$ denotes the number of tiles that $\gamma$ intersects.
\end{lemma}
\begin{lemma}
\label{lem:bord}
Let $\Tt$ be a locally finite tiling and  $\C$ be a compact arc connected subset of $\RR^d$. Then, for every locally
finite tiling 
$\Tt'$ satisfying  $R_{\Tt'} > R_\Tt$, and $K_\Tt < \L(\Tt', \C)\leq 
\L(\Tt,\C)$, we have:
\[\L(\Tt', \C) \leq (2K_{\Tt'} + 1) \frac{R_\Tt}{R_{\Tt'}} 
\L(\Tt,\C),\]
where $K_{\Tt'}$ is the constant defined in Lemma \ref{bound}. 
\end{lemma}

\begin{proof}[Proof of lemma \ref{lem:bord} ]
Given $y\in\RR^d$, set  
\[C_{y} := \{t\in\Tt'\mid t\cap B_{2R_{\Tt'}}(y)\neq\emptyset\}\quad
\text{and}\quad\widehat{C}_y:= \cup_{t\in C_y} t.\]
We construct a finite subset $\mathcal{Y}=\{y_i\}_{i=1}^{p}$ of 
 $\C$ as follows. First, fix any point of $\C$ to be $y_1$.
Next,  suppose that $y_1,\ldots,y_j$ have been chosen in such a way that, for
each $i\in\{1,\ldots,j\}$, the point 
$y_i$ does not belong to $\widehat{C}_{y_k}$ for every $k\in\{1,\ldots,i-1\}$. 
Then, if the sets $\{\widehat{C}_{y_i}\}_{i=0}^{j}$ cover $\C$, we set
$p=j$ and the construction is completed; otherwise, we choose any point in the
intersection of 
$\C$ and the complement of $\cup_{i=1}^{j}  \widehat{C}_{y_i} $ as $y_{j+1}$ and continue iterating the construction.
Since $\L(\Tt', \C)$ is finite (this is because $\Tt'$ is locally finite and $\C$ is compact) and, in each iteration we add at least one tile of $\Tt'$ intersecting $\C$ to the region covered by $\{\widehat{C}_{y_i}\}_{i=0}^{j}$, the construction stops after finitely many iterations. Thus, we have constructed
a set $\mathcal{Y}=\{y_i\}_{i=1}^{p}$ with the following properties:
\begin{enumerate}
\item $\cup_{y\in \mathcal{Y}} \widehat{C}_{y}$ covers $\C$;
\item $d(y_i,y_j)>2R_\T$ for every $i,j\in\{1,\ldots,p\}$ with $i\neq j$.
\end{enumerate}
Lemma \ref{bound} implies that for each $i\in\{1,\ldots,p\}$, the set
$C_{y_i}$ contains at most  $K_{\Tt'}$ tiles of $\Tt'$. From (1), we deduce that
 \begin{equation}
\label{eq:est1_p}
\L(\Tt', \C)\leq p K_{\Tt'}.
\end{equation}
By hypothesis, $p>1$.

\vspace{0.2cm}

Now define $\mathcal{B}=\{B_i = B_{y_i}(R_{\Tt'}-R_\Tt)\}_{i=1}^{p}$. From (2)
it is clear that the balls in $\mathcal{B}$  are pairwise disjoint and that the minimal
distance between two distinct balls in $\B$ is (strictly) greater than $2R_\Tt$.
Fix $i\in\{1,\ldots,p\}$. Since $p>1$, there is a point $z\in\C$ that
belongs to the complement of $B_{i}$. It follows
that there is a path $\gamma_i:[0,1]\rightarrow \C$ such that
$\gamma_i(0)=y_i$ and $\gamma_i(1)=z$. It is clear that the map $\eta:t\mapsto
\norm{y_i-\gamma(t)}$ is continuous, $\eta(0)=0$ and $\eta(1)>R_{\Tt'}-R_\Tt$.
By continuity, we deduce that there exists $t_i\in[0,1]$ such that
$\gamma_i([0,t_i])$ is included in $B_i$ and
$\norm{\gamma_i(t_i)-y_i}=R_{\Tt'}-R_\Tt$.  Applying Lemma
\ref{lem:diam_jordan_arc} we get
\[ R_{\Tt'}-R_\Tt\leq\diam{\gamma_i([0,t_i])}\leq 2R_\Tt  \L( \Tt,
\gamma_i([0,t_i]))\]
for all $i\in\{1,\ldots,p\}$. Adding these inequalities yields
\begin{equation}
\label{eq:est2_p0}
p (R_{\Tt'}-R_\Tt) \leq 2R_\Tt \sum_{i=1}^p \L( \Tt,\gamma_i([0,t_i])).
\end{equation}
Now, observe that since the image of $\gamma_i([0,t_i])$ is included in $B_i$
for each $i$, and the distance between $B_i$ and $B_j$ with $i\neq j$ is greater
than $2R_\Tt$,
each tile in $\Tt$ that intersects $\gamma_i([0,t_i])$ does not intersect
$\gamma_j([0,t_j])$ for every $j\neq i$. On the other hand, the image of
$\gamma_i$ is included in $\C$. Hence, we deduce that
\begin{equation}
\label{eq:est2_p01}
p (R_{\Tt'} -R_{\Tt})\leq 2R_\T\L(\Tt, \C).
\end{equation} 
Finally, combining  \eqref{eq:est2_p01} and \eqref{eq:est1_p}  we get 
\begin{equation}
\label{eq:est2_p}
\L(\Tt', \C) \leq 2K_{\Tt'}\frac{R_\Tt}{R_{\Tt'}-R_\Tt}   \L(\Tt, \C).
\end{equation}
To finish the proof, fix $c>0$ arbitrarily and consider the following two cases. First, suppose that $R_{\Tt'}/R_{\Tt} \leq 1 + c$. Since  $\L(\Tt', \C) \leq \L(\Tt, \C) $, it follows that 
\begin{equation}
\label{eq:menorquec}
\L(\Tt', \C)\leq (1+c)\frac{R_\Tt}{R_{\Tt'}}\L(\Tt, \C).
\end{equation}
Now suppose that $R_{\Tt'}/R_{\Tt} > 1 + c$. It is easy to check that 
\[(1+c)(R_{\Tt'} - R_{\Tt}) > c R_{\Tt'}.\] 
Replacing this inequality in  
\eqref{eq:est2_p} we get 
\begin{equation}
\label{eq:est2_p2}
\L(\Tt', \C) \leq 2K_{\Tt'} \frac{R_\Tt}{R_{\Tt'}}\left(1+\frac{1}{c}\right)\L(\Tt, \C).
\end{equation}
Combining \eqref{eq:est2_p2} and \eqref{eq:menorquec} yields
\[\L(\Tt', \C)\leq \max\left\{1+c,2K_{\Tt'}\left(1+\frac{1}{c}\right)\right\}\frac{R_\Tt}{R_{\Tt'}}\L(\Tt, \C).\]
An easy computation shows that the last bound is optimal when $c = 2K_{\Tt'}$ and the conclusion follows.
\end{proof}

\noindent $\bullet$ The third argument is a combinatorial one, and is related to the notion of {\it hierarchical decompositions} for substitution tilings as studied in \cite{ACG}, which we recall now. Let $\S$ be a substitution rule with dilation factor $\lambda$  and $\T$  be a tiling in  $\Omega_{\E,\P}$. The substitution map $\I_\S$ is defined as follows: Each tile $t$ in $\T$  is first  \emph{dilated} by $\lambda$ and then replaced 
by $O(\S_{p_i})$, where $t$ is a $\E$-copy of $p_i$ and $O$ is the isometry in $\E$ sending $p_i$ to $t$. The union of all these tiles is clearly a tiling, which we denote $\I_\S(\T)$. 
Similarly, we also can define $\mathcal{J}_\S^{(l)}:\Omega_{\E,\lambda^{(l+1)}\P}\rightarrow \Omega_{\E,\lambda^{l}\P}$ in a natural way. Recall that $\Omega_\S$ is defined as
\[\Omega_S =\bigcap_{k\geq 0}\I_{\S}^k(\Omega_{\E, \P}) .\]
It is plain to check that the map $\I_{\S}$ is onto when restricted  to
$\Omega_{\S}$. 
This implies that for each admissible tiling $\T$, there is a  sequence 
$(\T^l)_{l\geq 0}$ of tilings, called a \emph{hierarchical sequence} of $\T$, 
such that
\begin{itemize}
\item $\T^0 = \T$;
\item for each $l\geq 0$, $\T^l \in \Omega_{\E,\lambda^l \P}$ and 
$\J^{(l)}(\T^{l+1}) = \T^l.$
 
\end{itemize} 
\begin{rem}
\label{Khierar}
For every tiling $\T$ in  $\Omega_{\S}$ and all $l \geq 0$, it is easy to  see
that $r_{\T^l} = \lambda^l r_\T >0$ and $R_{\T^l} = \lambda^l R_\T < +\infty$.
Thus, $ K_{\T^l} = K_{\T}$ for all $l\geq 0$. 
\end{rem}

\begin{prop}[{\cite{ACG}}] \label{prop.hier}
Let $\T$  be an admissible tiling for $\S$ and let $(\T^l)_{l\geq 0}$ be a 
hierarchical sequence of $\T$ and suppose  that $\delta>0$ fits $\T$. For every 
subset $U\in \U_\delta$ such that $U$ and $\partial U$ are connected, 
there exists a finite collection 
$U_0,\dots , U_{m-1}$ of closed subsets of $U$ such that:
\begin{enumerate}[(i)]
\item $U_\T =: \cup_{i=0}^{m-1} U_i  =\cup_{t\in u_\T} t$ where
 $u_\T =\{t\in\T \mid   t\subset U\}$;  
\item   all $U_l$'s  have pairwise  disjoint interiors;
\item for each $l\in \{0, \dots, m-1\}$, $U_l$ is a union of  tiles in $\T^l$,  
and does not contain a tile in $\T^{l+1}$; 
\item $U$  does not contain a tile of $\T^m$. 
\end{enumerate}
The collection $\{U_0, \dots , U_{m-1}\}$ is called a {\rm hierarchical 
decomposition} of  $U$ associated with $\T$. 
Moreover, 
\begin{equation}
\label{eq:borde}
\quad \N(\T^l, U_l) \leq\, \|M_\S\|_1\L(\T^{l+1}, \partial U),
\end{equation}
for all $l\in\{0,\ldots,m-1\}$, where $\|M_\S\|_1$ is the maximum absolute
column
sum of the substitution rule $M$, 
and
\begin{equation}
\label{lem:lambda_m}
\lambda^{m - l - 1}\leq \frac{R_\T}{r_\T} \L(\T^l, \partial U),
\end{equation}
for all $l\in\{0,\ldots,m-1\}$.
    \end{prop}
    
    \begin{proof} Despite the fact that in \cite{ACG}, the proof of Proposition \ref{prop.hier} is given in the particular case when $U$ is a bounded connected domain of the plane, it works here exactly along the same lines.     \end{proof}

\noindent $\bullet$ The last argument is an algebraic one. It consists into applying Perron-Frobenius theory to the substitution matrix $M_\S$. Recall that 
we are are assuming that $M$ is primitive. Thus, Perron-Frobenius theorem asserts that 
$M_\S$ has a largest real eigenvalue $\mu  > 0$, the \emph{Perron eigenvalue},
which is simple and larger than one. Recall that $r(M)$  is the modulus of the second largest eigenvalue of $M$,  that is, 
\[r(M) = \max \{|\eta|\mid \eta \neq \mu\text{ is an eigenvalue of } M \}.\]
Now, if $v$ is a right-eigenvector associated with the eigenvalue $\mu$ with positive coefficients and we let $e$ be the unit vector $(\frac{1}{\sqrt{n}}, \dots, \frac{1}{\sqrt{n}})$, then we have the following well-known consequence of Perron-Frobenius Theorem, 
see for instance \cite{HornJohnson}. 
\begin{prop}\label{PF}
For all $\rho > r(M)$,   there exist $K = K(\rho)>0$   and $\alpha>0$ such
that
\begin{equation}
\label{eq.pf1}
\Vert M_\S^le - \alpha\mu^l v \Vert_1\leq K\rho^l
\end{equation}
for all integer $l>0$.  
\end{prop}
Recall that $\mathcal{Q}$ is the set of types of prototiles. Given $q_i,q_j \in \mathcal{Q}$, $m_{i,j}$ counts the number 
of tiles of type $q_i$ in $S_{p}$ where $p$ is any prototile in $q_j$. The relation with Perron-Frobenius theory comes from the fact
that if $m_{i,j}^{(l)}$ is the corresponding element of $M_\S^l$, then $m_{i,j}^{(l)}$ counts the number of tiles of type $q_i$ in a prototile of type $q_j$ after applying the substitution to it $l$ times. 

\begin{proof}[Proof of Theorem \ref{lem:key1}]

Let 
$\{U_0,\ldots,U_{m-1}\}$ be the
hierarchical decomposition of $U$ given by Proposition \ref{prop.hier}. 
Since $\N(\T, U)\ge 1$, we have that $m\ge 1$.
Recall that $\mathcal{Q}=\{q_1,\ldots,q_n\}$ denote the set of
$\S$-equivalent classes of tiles in $\P$. 
The $\S$-equivalent relation can be extended to the set of tiles 
$\E$-equivalent to some tile in $\lambda^l\P$, for every integer $l\ge 0$. 
We have  that a tile $\lambda^lt$ is $\S$-equivalent to $\lambda^lp$ for some
$p\in \P$ if and only if $t$ is  $\S$-equivalent to $p$. 
Thus, the number of $\S$-equivalence classes of the tiles in $\P$ and
$\lambda^l\P$ is the same.
Denote by
$\{\lambda^lq_1,\ldots,\lambda^lq_n\}$ the set of $\S$-equivalence classes of
tiles in $\lambda^l\P$ and by
$\{\lambda^lp_1,\ldots,\lambda^lp_n\}$  a  set of representatives of these
classes.

Observe  that 
\begin{equation}
\label{eq:syn_1}
\N(\T, U) = 
\sum_{l=0}^{m-1}\sum_{t^l\in U_l} \N( \T, t^l)
= 
\sum_{l=0}^{m-1} \sum_{i=1}^n\sum_{\stackrel{t^l\in U_l}{t^l\in
\lambda^l q_i}} 
\N(\T, t^l, \lambda^l q_i),
\end{equation}
where  $\N(\T, t^l, \lambda^l q_i)$ denotes the number of tiles  of $\T$
included in the tile $t^l\in \lambda^l q_i$. Using the  matrix $M_\S$, we deduce
\[\N(\T, t^l, \lambda^l q_i) = \sum_{j=1}^n m_{i,j}^l,\] 
where  $m_{i,j}^l$ denotes the $(i,j)$-element of the matrix $M_\S^l$.
Moreover, 
it is straightforward to check, using the fact that any tile in  $\lambda q_i$
is tiled by $O(\RR^d)$-copies of tiles in $ q_j$, for $j=1,\ldots,n$, and that
$M_\S$ counts these tiles, that 
\[w = (\mu_d(p_1), \dots, \mu_d(p_n))\]
is a right-eigenvector for $M_\S$ with eigenvalue $\lambda^d$, 
which is the Perron-Frobenius eigenvalue of $M_\S$. Thus, as a direct corollary
of Proposition \ref{PF}, we get:
\begin{cor}
\label{cor:PF}
For all $\rho
> r(M_\S)$, there are $K_0$ and $\alpha>0$ such that for all $\, l\in \{0,
\dots, m\}$, all $
i\in \{1, \dots, n\}$, we have
\[\vert\N(\T, t^l, \lambda^l p_i)\, -\,  \alpha\mu_d (\lambda^l p_i)\vert \,
\leq \, K_0 \rho^l.\]
\end{cor}

On the other hand
\begin{equation}
\label{eq:syn_2}
\mu_d (U_\T)\, =\, \sum_{l = 0}^{ m-1}\sum_{\stackrel{ t^l\in \T^l,}{t^l\subset
U_l}} \mu_d(t^l) \, =\,\sum_{l = 0}^{ m-1} \sum_{i=1}^n\sum_{\stackrel{ t^l\sim
\lambda^l p_i\in \T^l,}{t^l\subset U_l}} \mu_d(\lambda^lp_i).
\end{equation}
Multiplying \eqref{eq:syn_2} by $\alpha$ and then subtracting  to \eqref{eq:syn_1} we get 
\begin{equation}
\label{eq:syn_2_1} \N(\T, U)  - \alpha\mu_d (U_\T)\, =    \,
\sum_{l = 0}^{ m-1} \sum_{i=1}^n\sum_{\stackrel{ t^l\sim \lambda^l p_i\in
\T^l,}{t^l\subset U_l}} (\N(\T, t^l,\lambda^l  q_i)- \alpha \mu_d(\lambda^l
p_i)).
\end{equation}
Applying Corollary \ref{cor:PF} to \eqref{eq:syn_2_1} yields
\begin{equation}
\label{eq:syn_3}
\vert \N(\T, U)  - \alpha\mu_d (U_\T) \vert \leq 
\sum_{l = 0}^{ m-1} \sum_{i=1}^n \sum_{\stackrel{ t^l\sim \lambda^l p_i\in
\T^l,}{t^l\subset U_l}} K_0 \rho^l
=  \sum_{l = 0}^{m-1} \N(\T^l, U_l)K_0\rho^l.
\end{equation}
After combining \eqref{eq:borde}  and
\eqref{eq:syn_3}, we get
\begin{equation}
\label{eq:syn_4}\vert \N(\T, U)  - \alpha\mu_d (U_\T) \vert
\leq 
K_0 \|M\|_1 \sum_{l=0}^{m-1}\L(\T^{l+1},\partial U) \rho^l.
\end{equation}

We want to apply Lemma \ref{lem:bord} to give an upper bound of
$\N(\Tt^{l+1},\partial U)$ in terms of $\N(\Tt,\partial U)$ for all
$l\in\{0,\ldots,m-1\}$.
First, observe that  the sequence $\N(\T^l, \partial U), l\in \NN,$ is
decreasing. Define $l_0$ to be the largest $l\in\{1,\ldots,m\}$ such that 
$\N(\T^l, \partial U) > K_{\T}$. 
Now, we split the sum in \eqref{eq:syn_4} into two parts
\begin{equation}
\label{eq:syn_split}
\sum_{l=0}^{l_0-1}\L(\Tt^{l+1},\partial U) \rho^l + 
\sum_{l=l_0}^{m-1}\L(\Tt^{l+1}, \partial U) \rho^l.
\end{equation}
We apply Lemma \ref{lem:bord} to each term in the first sum, 
which gives
\[
\L(\Tt^{l+1}, \partial U)\leq
(2K_{\Tt^{l+1}}+1)\frac{R_\Tt}{R_{\Tt^{l+1}}}\L(\Tt, \partial U)
\]
for all $l\in\{0,\ldots,l_0-1\}$. Since $K_{\T^l} = K_{\T}$ and 
$R_\Tt/R_{\Tt^l}=\lambda^{-l}$ for every $l\geq 0$, it follows that
\begin{equation}
\label{eq:syn_first_sum}
\sum_{l=0}^{l_0-1}\L(\Tt^{l+1}, \partial U) \rho^l \leq  
(2K_\Tt+1) \lambda^{-1} \L(\Tt, \partial U)
\sum_{l=0}^{l_0-1}\left(\frac{\rho}{\lambda}\right)^l. 
\end{equation}

Now we estimate the second sum in \eqref{eq:syn_split}. Since $\rho < \lambda$
and 
the sequence $\N(\Tt^l, \partial U), l\in \NN,$
is decreasing, we have
\begin{equation}
\label{eqNl0}
\sum_{l=l_0}^{m-1}\L(\Tt^{l+1},\partial U) \rho^l
\leq\L(\Tt^{l_0}, \partial U)\sum_{l=l_0}^{m-1} \lambda^l \leq 
K_\T
\left(\frac{\lambda^m-\lambda^{l_0}}{\lambda-1}\right).
\end{equation}

From  \eqref{lem:lambda_m}
we get   $\lambda^{m}\leq \lambda R_\Tt/r_\Tt \L(\T, \partial U)$. 
Setting $N_\Tt:= (K_\Tt R_\Tt/r_\T)(\lambda/(\lambda-1))$ and replacing into 
\eqref{eqNl0} we get 
\begin{equation}
\label{eq:syn_6}
\sum_{l=l_0}^{m-1} \L(\Tt^{l+1},\partial U) \rho^l\leq  N_\Tt\L(\Tt, \partial
U).
\end{equation} 
Combining 
\eqref{eq:syn_first_sum} and \eqref{eq:syn_second_sum}, we get

\begin{equation}
\label{eq:syn_second_sum}
\sum_{l=0}^{m-1} \L(\Tt^{l+1},\partial U) \rho^l\leq \left( (2K_\Tt+1)
\sum_{l=0}^{l_0-1}\left(\frac{\rho}{\lambda}\right)^l + N_\Tt\right) \L(\T,
\partial U).
\end{equation}
Using \ref{eq:syn_4} we conclude that 

\begin{equation}
\label{eq:syn_4}
\vert \N(\T, U)  - \alpha\mu_d (U_\T) \vert
\leq 
K_0 \|M\|_1(2K_\Tt+1)\left( (2K_\Tt+1) 
\sum_{l=0}^{l_0-1}\left(\frac{\rho}{\lambda}\right)^l + N_\Tt\right) 
\L(\T,\partial U).
\end{equation}
Since we can choose $r(M)<\rho<\lambda$, for all $\delta$ that fits with $\T$,
there exists a constant $\hat K$ which depends only on $\T$ such that 
\begin{equation}
\label{eq:syn_5}
\vert \N(\T, U)  - \alpha\mu_d (U_\T) \vert
\leq 
\hat K \L(\Tt, \partial U),
\end{equation}

On the other hand, we have 
\begin{equation}
\label{eq:syn_6}
\vert \mu_d(U)-  \mu_d (U_\T) \vert
\leq 
(2R_\T)^d \L(\Tt, \partial U),
\end{equation}
for all $U\in \U_\delta$.
Combining these last two equations, we get that  for all $\delta$ that fits with $\T$, there exists a constant $K$ which depends only on $\T$ such that:
\begin{equation}
\label{eq:syn_5}
\vert \N(\T, U)  - \alpha\mu_d (U) \vert
\leq 
K \L(\Tt, \partial U).
\end{equation}

\end{proof}

\noindent \textbf{Acknowledgments.} 
After a first version of this work was sent to  Arxiv, we learned from M. Baake that 
D. Frettl\"oh and A. Garger \cite{Frettloh} obtained a version of Theorem \ref{thm:main} when $d=2$.

\noindent This work is part of the project {\it CrystalDyn} supported
by the "Agence Nationale de la Recherche" (ANR-06-BLAN- 0070-01). 
D. Coronel and J. Aliste-Prieto respectively ackwnoledge support from Fondecyt post-doctoral Grants 3100092 and  3100097. 

\noindent The authors are very grateful to Dong Ye who, very generously, introduced them to  
his work with Tristan Rivi\`ere. 


\end{document}